\documentclass{article}      
\usepackage{amssymb}
\usepackage{amsmath}
\usepackage{latexsym}
\usepackage[mathscr]{euscript}
\usepackage{graphicx}
\usepackage{amscd}
\usepackage{amsthm} 
\usepackage{fullpage} 
\usepackage{wrapfig}
\newtheorem{theorem}{Theorem}
\newtheorem{corollary}[theorem]{Corollary}
\newtheorem{lemma}[theorem]{Lemma}  
\newtheorem{proposition}[theorem]{Proposition}

\newtheorem{definition}{Definition}

\newcommand{\bp}{\mathbb{ P}}
\newcommand{\bz}{\mathbb{Z}}
\newcommand{\br}{\mathbb{R}}
\newcommand{\bq}{\mathbb{Q}}
\newcommand{\bh}{\mathbb{H}}

\newcommand{\ce}{\mathcal{E}}

\newcommand{\cl}{\mathcal{L}}
\newcommand{\cs}{\mathcal{S}}
\newcommand{\cg}{\mathcal{G}}

\newcommand{\cp}{\mathcal{P}}
\newcommand{\ch}{\mathcal{H}}

\newcommand{\cf}{\mathcal{F}}

\newcommand{\ca}{\mathcal{A}}

\newcommand{\hk}{\hookrightarrow}
\newcommand{\bg}{\bigskip}
\newcommand{\med}{\medskip}

\newcommand{\bfl}{\begin{flushleft}}
\newcommand{\efl}{\end{flushleft}}

\newcommand{\xr}{\xrightarrow}

\newcommand{\G}{\Gamma}

\newcommand{\ltm}{LM^{-TM}}

\newcommand{\pud}{P^{Ad}}

\newcommand{\cpm}{\cp(M)}
\newcommand{\sgp}{ \Sigma^\infty (\cg (P)_+) }

\newcommand{\fsm}{\Sigma^\infty_M}
\newcommand{\cgsp}{hAut(\fsm (P_+))}

\begin{document}

  \title{Gauge theory and string topology}  
  \author{Ralph L. Cohen \thanks{The first author was partially supported by a  grant  from the NSF.} \\ Department of Mathematics \\Stanford University \\ Bldg. 380 \\ Stanford, CA 94305, USA \and  John D.S Jones \\ Mathematics Institute \\Zeeman Building \\ Warwick University\\ Coventry, CV4 7AL, UK }
\date{\today}
\maketitle  
 \begin{abstract}   Given a principal bundle over a closed manifold,  $G \to P  \to M$,  let $P^{Ad} \to M$ be the associated adjoint bundle.  In \cite{gruhersalvatore} Gruher and Salvatore showed that the Thom spectrum $(P^{Ad})^{-TM}$ is a ring spectrum whose corresponding    product  in homology  is a Chas-Sullivan type string topology product.   We refer to this spectrum as the ``\it{string topology spectrum of P}", $ \cs (P)$.   \rm   In the  universal  case  when $P$ is contractible,  $\cs (P) \simeq LM^{-TM}$  where $LM$ is the free loop space of the manifold.  This ring spectrum was introduced by the authors  in \cite{cohenjones}  as a homotopy theoretic realization of the Chas-Sullivan string topology of $M$.  The main purpose of this paper is to introduce an action of the gauge group of the principal bundle, $\cg (P)$ on the string topology spectrum $\cs(P)$, and to study this action in detail.  Indeed we study the entire group of units and the induced representation $\cg (P) \to  GL_1(\cs (P))$. 
   We show that this group  of units is the group of homotopy automorphisms of the fiberwise suspension spectrum of $P$.   More generally we describe the homotopy type of the group of homotopy automorphisms of any $E$-line bundle for any ring spectrum $E$.    We import some of the basic ideas of gauge theory, such as the action of the gauge group on the space of connections to the setting of   $E$-line bundles  over a manifold, and do explicit calculations.    We end by discussing a functorial perspective, which describes
 a sense in which the string topology spectrum $\cs (P)$  of a principal bundle is  the ``\sl{linearization}" \rm of the gauge group $\cg (P)$. 
 \end{abstract}

 \tableofcontents

 \section*{Introduction}      Let $M$ be a closed, smooth, connected $n$-dimensional manifold with a basepoint $x_0 \in M$.   Let $P$ be a principal $G$-bundle over $M$ where $G$  is a topological group.  Let $\cg (P)$ be the  gauge group of $P$.  That is,  $\cg (P)$ is the group of continuous automorphisms of the principal bundle $P$ that    cover  the identity map of $M$.     The gauge group $\cg (P)$ has the following description.   Let $\pud$ be the associated adjoint bundle,
 $$
 G \to \pud = P \times_G G  \to M
 $$
 where $G$ acts on itself by conjugation.  $\pud$ is a fiberwise group, in the sense that there is an associative pairing of bundles over $M$,  
 $$
 \begin{CD}
 G \times G    @>mult >>  G \\
 @VVV    @VVV \\
 \pud \times_M \pud @>\mu >>  \pud \\
 @VVV   @VVV \\
 M  @>>=> M.
 \end{CD}
 $$
   The space of  sections $\G_M(\pud)$ is   a group, and is clearly isomorphic to the gauge group $\cg (P)$.  (See \cite{atiyahbott} for a general discussion of the topology of the gauge group.)   The goal of this paper is to relate the topology of such gauge groups to  string topology as introduced by Chas and Sullivan \cite{chassullivan}.
 
 To be more precise,  recall that since $\pud \to M$ is a fiberwise group, then if $(\pud)^{-TM}$   denotes the Thom spectrum of the virtual bundle defined by pulling back $-TM \to M$ over $\pud$,   then Gruher and Salvatore \cite{gruhersalvatore} showed that $(\pud)^{-TM}$ is a ring spectrum.  Furthermore if $M$ is oriented,    that  product structure induces an algebra structure on $H_{*+n}(\pud)$, very much analogous to the Chas-Sullivan string topology algebra structure on $H_{*+n}(LM)$, the (shifted) homology of the free loop space.  We refer to $(\pud)^{-TM}$ as the ``string topology spectrum" $\cs (P)$.  
 
  Actually  the relation between the string topology spectrum $\cs (P)$ and the  Chas-Sullivan structure on $LM$ is more than an analogy.  Indeed the  Chas-Sullivan structure is a special case of this construction.  Namely,  let $\cpm \to M$ be any principal bundle where the total space $\cpm$ is contractible. We view $\cpm \to M$ as a \sl universal \rm principal bundle over $M$. Such a bundle has the fiber homotopy type of the path loop fibration $\Omega M \to P(M) \xr{ev} M$, where $P(M) = \{ \alpha : [0, 1] \to M \quad \text{such that} \quad \alpha (0) = x_0 \in M \}$.  Here $ev (\alpha) = \alpha (1)$.  Then it is well known that the corresponding adjoint bundle is homotopy equivalent to the free loop space,  $\cpm^{Ad} \simeq LM$.  So $(\cpm^{Ad})^{-TM} \simeq LM^{-TM}$ which in this case is an $E_2$ ring spectrum that reflects the Chas-Sullivan Gerstenhaber structure on $H_{*+n}(LM)$.   
 
 \med
 In  this paper we introduce an action of the gauge group $\cg (P)$ on the string topology spectrum $\cs (P)$.  This defines a representation
 $\cg (P) \to   GL_1(\cs (P))$.  We study this representation in detail.   In so doing we will show that the group of units $GL_1(\cs (P))$  is equivalent to  the group of homotopy automorphisms of the fiberwise suspension spectrum $\Sigma^\infty_M(P_+)$.   This is a parameterized   spectrum whose fibers are rank one, free modules over the    suspension spectrum of the group,  $\Sigma^\infty (G_+)$.  Indeed, given any ring spectrum $E$,  we consider parameterized spectra of free, rank one $E$-modules,  ``$E$-line bundles", and   we study the homotopy type  its group of  homotopy automorphisms,  $hAut (\ce)$.   The above representation of the gauge group can then be viewed as a homomorphism of group-like monoids,     $$
 \iota : \cg (P) \to  hAut (\Sigma^\infty_M (P)).
 $$

 \med
 
 We now state our results more carefully.  First we establish some terminology.

 Let $E $ be a ring spectrum.   The notion of  a parameterized spectrum of free, rank one $E$-modules (``$E$-line bundle") was described and studied in  \cite{5author},      We define the spectrum $End (\ce)$ to be the (ordinary) spectrum of endomorphisms of $\ce$, living over the identity of $M$.  This spectrum will be defined more carefully below.  
  We then define 
 $$
 hAut (\ce) = GL_1(End (\ce)).
 $$
 
 The following is the first result, which is an application of fiberwise Poincare-Atiyah duality.
 
 \begin{proposition}\label{stringend}
 Let $G \to P \to M$ be a principal bundle, and let $\Sigma^\infty (G_+) \to \Sigma^\infty_M (P_+) \to M$ be the induced  $\Sigma^\infty(G_+)$-line bundle given by
 the fiberwise suspension spectrum of $P$.   Then there is an equivalence of ring spectra,
 $$
 End (\Sigma^\infty_M (P_+)) \simeq \cs (P) \simeq (\pud)^{-TM}.
 $$  
 \end{proposition}

 \begin{corollary}\label{gaugestring}  
 $$   hAut(\Sigma^\infty_M (P_+) ) \simeq GL_1(\cs(P)).$$
 
 In particular if $\cpm \to M$ is a universal principal bundle (i.e $\cpm$ is contractible),  then
 $$
 GL_1(LM^{-TM}) \simeq hAut (\cpm).
 $$
 \end{corollary}

 We use this result in the proof of the first of our main theorems, which identifies the homotopy types of these   groups.
 
 \med
 Recall  from \cite{5author} that if $E$ is a ring spectrum then the  the classifying space   $BGL_1(E)$  classifies
$E$-line bundles.   Given a such an  $E$-line bundle  $E \to \ce \to M$, let $Map_\ce (M, BGL_1(E))$ denote the path component of the mapping space consisting of maps classifying $\ce$.   
 
 \med
 \begin{theorem}\label{main1}
 Let $E \to \ce \to M$ be an $E$-line bundle  over $M$. Then there is an equivalence of group-like $A_\infty$-spaces,
 $$
 hAut (\ce) \xr{\simeq}  \Omega Map_\ce(M, BGL_1(E)).
 $$
 \end{theorem}

 \med
 Let $G \to P \to M$ be a principal bundle.  This bundle is classified by the homotopy class of a map $f_P : M \to BG$.  Let $Map_P(M, BG)$ denote the path component of the mapping space consisting of maps of this homotopy class.
  Now notice that the group $G$ has a natural action on the suspension spectrum, $G _+ \wedge \Sigma^\infty (G_+) \to \Sigma^\infty (G_+)$, which is analogous in the category of spectra to the representation of  group  on its group ring.  This defines a map of group-like monoids, 
 $$
  G \xr{\iota}  GL_1(\Sigma^\infty (G_+)) \quad \text{and of classifying spaces} \quad  BG \xr{B\iota} BGL_1(\Sigma^\infty(G_+)).
 $$
  A special case of this is the group-like monoid   $\Omega M$.  Its suspension spectrum $\Sigma^\infty (\Omega M_+)$ has the structure of an associative ring spectrum.  There is a natural equivalence of $M$ with the classifying space  $M \simeq B(\Omega M)$, thereby inducing the map $\iota : M \to BGL_1(\Sigma^\infty (\Omega M_+)$.  By abuse of notation we write $Map_P(M, BGL_1(\Sigma^\infty(G_+))$ to be the component of the mapping space in the image of $B\iota$. 
 
 We then  have the following corollary.

 \med
 \begin{corollary}\label{cor1}
 One can identify the homotopy type of the group of units of the string topology spectrum as the loop space of the mapping space,
 $$
 GL_1(\cs (P)) \simeq \Omega Map_P(M, BGL_1(\Sigma^\infty (G_+)).
 $$
 In particular,
 $$
 GL_1(LM^{-TM}) \simeq \Omega Map_\iota (M, BGL_1(\Sigma^\infty (\Omega M_+)).
 $$
 \end{corollary}

  \med
  We will then use this description of the group of units of the string topology spectrum to perform explicit homology calculations.

 \med
 
 Much of the rest of the paper considers the comparison between the gauge group of a principal bundle, $\cg (P)$, with the group of homotopy automorphisms of its suspension spectrum, $hAut (\Sigma^\infty_M(P_+))$.  There is a natural suspension map that yields a homomorphism of group-like monoids,
\begin{equation}\label{atbott1}
\sigma :  \cg (P) \to \cgsp.
\end{equation}
 
  In \cite{atiyahbott}  Atiyah and Bott gave a beautiful exposition of the topology of the gauge group.  In particular they gave a proof of a result originally observed by
  Gottlieb, saying that the classifying space of the gauge group is a mapping space,
  \begin{equation}\label{abgauge}
  B\cg (P) \simeq Map_P(M, BG).
  \end{equation}
  We observe that Theorem \ref{main1} above is an analogue of this result.  In particular we make the following observation:
  
  \begin{proposition} \label{iota} With respect to the equivalence given in Theorem \ref{main1} and the Atiyah-Bott equivalence (\ref{atbott1}),   the  stabilization map $B\sigma : B\cg (P) \to B \cgsp \simeq BGL_1(\cs (P))$ is  induced  by the map $\iota$ described above: 
  $$
  \iota_*: Map_P(M, BG) \to Map_P(M, BGL_1(\Sigma^\infty (G_+)).
  $$
  \end{proposition}
  
  \med
The most fundamental construction in gauge theory  is the action of the gauge group $\cg (P)$ on the space of connections $\ca (P)$.   Since $\ca (P)$ is contractible, the homotopy orbit space, which we denote by $\ca (P) // \cg (P)$,  is a model for the classifying space of the gauge group, $B\cg (P)$.  Our next result is an analogue of this in the setting of $E$-line bundles. 

  For an $E$-line bundle  $E \to \ce \to M$,  we define a category of connections, $Conn (\ce)$, where a connection is a functorially defined $E$-linear transport operator on $\ce$.  This is actually an $\infty$-category, which is an $\infty$- groupoid, where the morphisms are equivalent to the gauge group $\cg (P)$.  This is the relevant homotopy theoretic notion of the gauge group of an $E$-line bundle   ``acting" on the space of connections. We then prove the following.
  
  \begin{theorem}\label{main2}
  The geometric realization of the category of connections on $\ce$  is  homotopy equivalent to the classifying space of the   group of homotopy automorphisms
  $$
 |Conn (\ce)| \simeq  B hAut(\ce).
 $$
 In particular when $\ce = \fsm (P_+)$, the fiberwise suspension spectrum of a principal bundle, we then have that geometric realization of the category of connections
 is the classifying space of the group of units of the string topology spectrum,
 $$
 |Conn (\fsm (P_+))| \simeq BGL_1(\cs (P)).
 $$
 \end{theorem}

 \med
 Now a quick consequence of the Atiyah-Bott result (\ref{atbott1}) is that for a  universal principal bundle, $\Omega M \to \cp(M)\to M$, the gauge group has the homotopy type,
   $$
 \cg (\cp(M)) \simeq \Omega HE(M). 
 $$
 where $HE(M)$ is the group-like monoid of self homotopy equivalences of $M$.  The basepoint of $HE(M)$ is the identity map.  
 Therefore,  If $M$ is acted upon by a group $K$,  then the induced map $K \to  HE(M)$
yields the following.
 
 \med
 \begin{proposition}\label{equivariant}
 Given a group action of $K$ on $M$, there is an induced representation of the loop group on the string topology spectrum,
 $$
 \Omega K \to GL_1(LM^{-TM}).
 $$
 \end{proposition}
 
 The induced ring homomorphism in rational homology, $$H_*(\Omega HE(M); \bq) \to H_{*+n}(LM; \bq)$$ is equivalent to the homomorphism introduced and studied by F\'elix and Thomas in \cite{felixthomas}.    
 
 \med
 We end by taking a functorial view of the relationship between the gauge group of a principal bundle and its string topology.   To make this more precise, consider the category of   ``based spaces over $M$",   $Spaces_M$.  This category has objects given by continuous maps $f : X \to M$ that are equipped with sections, $s : M \to X$ such that $f\circ s = id_M$.  Morphisms in this category are maps $g : X \to Y$ that cover the maps to $M$ and preserve the sections.    Now given a principal bundle $G \to P \to M$, consider the contravariant functors,
 
$$
 \cg : Spaces_M  \to Spectra  \quad \text{and} \quad
 \cs : Spaces_M  \to Spectra $$
 
 defined by
 $$
 \cg (X \xr{f} M) = \Sigma^\infty (\cg (f^*(P))_+)
 \quad 
 \text{and} 
 \quad
 \cs  (X \xr{f} M)  = \cs (f^*(P)).
 $$
   Our goal is to prove the following result about these functors.
 
 \med
 \begin{theorem}\label{main3}  Given a  principal bundle $G \to P \to M$, the map of ring spectra  
 $ 
\sigma : \sgp \to  \cs (P) $ generalizes to give a natural transformation  between the gauge group functor $\cg : Spaces_M \to Spectra$ and the string topology functor, $\cs : Spaces_M \to Spectra$.   Furthermore,  this natural transformation satisfies the following properties.
\begin{itemize}
\item $\cs$ is ``linear", in the sense that it takes homotopy cocartesian squares to homotopy cocartesian squares.   If one applies homotopy groups, the induced functor $\pi_* (\cs (-)) : Spaces_M \to Graded \, Abelian  \, Groups$, is  a generalized cohomology theory defined on the category $Spaces_M$.
\item $\sigma : \cg \to \cs$ is the universal linear approximation to $\cg$, in the sense that if $\tau : \cg \to \cl$ is a natural transformation to any linear
contravariant homotopy functor, $\cl : Spaces_M \to Spectra$,  then there is a unique natural transformation $\bar \tau : \cs \to \cl$ making the following diagram commute:
   $$
   \begin{CD}
   \cg @>\tau >>  \cl \\ 
   @V\sigma VV     @VV=V \\
   \cs  @>>\bar \tau >  \cl
   \end{CD}
   $$
   \end{itemize}
   \end{theorem}
   
   The fact that the string topology functor is the ``linearization" of the gauge group functor, suggests the existence of a ``Goodwillie tower" \cite{goodwillie}
   interpolating  between these two functors.  This idea is being investigated by C. Malkiewich \cite{cary}.

   \med
   This paper is organized as follows.   In section one we collect various  properties of parameterized spectra, including a Poincar\'e-Atiyah duality theorem, and we prove Proposition \ref{stringend}  and Corollary \ref{gaugestring} as stated above.   In section 2 we prove Theorem \ref{main1}, Corollary \ref{cor1},  and Proposition \ref{iota}.  We then   apply them to perform explicit calculations of the homology of the group of units of   string topology spectra.    In section 3 we describe the category of connections on an $E$-line bundle over a manifold,    and prove Theorem \ref{main2}.      In section 4 we prove   Proposition \ref{equivariant}.  As an example we   then calculate the induced map in homology of the representation of the loop group $\Omega SU(2)$ on the string topology spectrum of $S^4$,  $(LS^4)^{-TS^4}$.  Finally in section 5 we describe the functorial approach to the relationship between the gauge group and the string topology spectrum, and prove Theorem \ref{main3}. 
   
   \med
   The authors are grateful to Nitu Kitchloo, John Klein, and Jack Morava for helpful conversations about this work. 
 
 \section{Parameterized spectra,  Poincar\'e duality, and the loop space}
 
 In this section we recall some basic constructions and duality results from the theory of parameterized spectra.    The most comprehensive reference for parameterized spectra is the book by May and Sigurdsen \cite{maysigurd}.
 
 Let $X$ be a topological space of the homotopy type of a CW complex.  The theory of parameterized spectra involves constructions that take place in the category $Spaces_X$ of based (retractive) spaces over $X$.    
  In this category there are standard constructions such as fiber product, $Y \times_X Z$,  fiberwise (reduced) suspension $\Sigma_XY$, and fiberwise smash product, $Y \wedge_X Z$.  In particular, the fiberwise smash product is defined to be the 
 pushout of the diagram
$$
X \leftarrow Y\cup_X  Z \to  Y \times_X Z  \, .
$$

\begin{definition} \label{fibered}
A {\it parameterized spectrum} 
$ {\cal E}$ over $X$
consists of  objects ${\cal E}_j \in {\cal R}_X$ for $j \in {\Bbb N}$ together
with (structure) maps
$$
\Sigma_X {\cal E}_j \to {\cal E}_{j+1} \,  ,
$$
for each $j \ge 0$.
A {\it morphism} ${\cal E}\to {\cal E}'$ is given by   
maps ${\cal E}_j \to {\cal E}'_j$ which are compatible with the
structure maps. 
\end{definition}

  ${\cal E}$ is said to be {\it fibrant} if the adjoints to the structure
maps, $\ce_j \to \Omega_X \ce_{j+1}$  are weak homotopy equivalences.
By usual techniques, any parameterized  spectrum ${\cal E}$ can be converted into a fibrant
one ${\cal E}^f$ in which 
$$
{\cal E}^f_j\,\, := \,\, \underset{n}{\text{hocolim\, }} 
\Omega^n_X {\cal E}_{j+n} \, ,
$$
where  
$\Omega^n_X $ denotes the $n$-fold  fiberwise based loop space.   $\ce^f$ is   called  the {\it fibrant 
replacement} of $\ce$.

A morphism ${\cal E}\to {\cal E}'$ is a 
{\it weak equivalence} if the associated morphism of 
fibrant replacements  ${\cal E}^f \to ({\cal E}')^f$ 
is a {\it level-wise} weak equivalence: 
for each $j$, the map ${\cal E}_j^f \to ({\cal E}')_j^f$ 
is required to be weak equivalence in ${\cal R}_X$. 

\med

Probably the most important example of a parameterized spectrum for our purposes, is that of the  \sl fiberwise suspension spectrum. 
\rm 
 For  $Y$ an object in  $Spaces_X$ ,  we denote by $\Sigma^\infty_X Y$
the parameterized  spectrum over $X$ given by the collection $\Sigma^j_X Y$ of
iterated fiberwise suspensions of $Y$.     

\med
If $f : Z \to X$ is any map, then one constructs an object $Z_+$ in $Spaces_X$ by  adding a ``disjoint fiberwise basepoint".  That is,
$$
Z_+ = Z \sqcup X
$$
and has structure maps $f \sqcup id : Z \sqcup X \to X$, and $s : X \to Z \sqcup X$ is the inclusion of the righthand summand.  

\med
The primary way parameterized spectra arise in our study is in the context of fibrations over manifolds.  So suppose
$$
F \to E \xr{p} M
$$
is a fibration over a connected, closed manifold $M$.  Then the fiberwise suspension spectrum
$$
\Sigma^\infty_M(E_+) \xr {p} M
$$
is a parameterized spectrum, whose fiber spectrum is the (ordinary) suspension spectrum $\Sigma^\infty (F_+).$

\med
The following is a construction that is important in the statement of Poincar\'e duality in the context of parameterized spectra.
Given a parameterized spectrum $\ce \to X$,  and a $k$-dimensional vector bundle $\xi \to X$, we define a new parameterized spectrum
$\ce_\xi \to X$ in the following way.  Let $S^k \to S(\xi) \to X$ be the associated sphere bundle, defined by taking the fiberwise one-point compactification. 

\begin{definition}\label{twist}
Define the parameterized spectrum  $\ce_\xi $ to be
$$\ce_\xi  = S(\xi) \wedge_X \ce \to X.
$$
\end{definition}

 Notice this is a parameterized spectrum whose fiber is given by the smash product $S^k \wedge \cf \cong \Sigma^k\cf$.  
 We remark that by standard techniques,   one can define this construction for virtual bundles $\xi$ as well. See  \cite{maysigurd} for details, but note that the notation   in  this reference  is different than that used here.
 
 \med
 Now as is well known,  a parameterized spectrum $\ce \to X$, determines (reduced) homology and cohomology theories defined on the category $Spaces_X$.
\begin{definition} \label{homology}  Given an object $Y$ of $Spaces_X$,  the  homology spectrum is defined to be the (ordinary) spectrum  
 \begin{equation}
\ce(Y) =  \ce \wedge_X Y/Y.
 \end{equation}
The homology groups are  defined to be the homotopy groups of $\ce(Y)$.
 \end{definition}
 
 This notation needs some explanation.  In this definition, $\ce \wedge_X Y$ is the fiberwise smash product of $Y$ with the fibrant replacement $\ce^f$.  Furthermore $\ce \wedge_X Y/Y$ means the mapping cone of the resulting ``fiberwise basepoint",  $\sigma : Y \to \ce \wedge_X Y$.

 \med
 The following is an easy exercise.
 
 \begin{lemma}\label{thom}  Suppose  $\ce$ is the parameterized spectrum over $X$ given by the fiberwise suspension spectrum of a fibration, $F \to E \xr{p} X$.  That is,
  $\ce = \Sigma^\infty_X(E_+).$  Let $\xi \to X$ be a virtual vector bundle.  Then   the twisted homology spectrum $\ce_\xi (X)$ is equivalent to the 
 Thom spectrum of the pullback bundle $p^* \xi \to E$, which we denote by $E^\xi$.  That is, there is an equivalence,
 $$
 \ce_\xi (X) \simeq E^\xi.
 $$
 \end{lemma}
 
 There is a similar definition of the associated twisted cohomology.
 \begin{definition}\label{cohomology}
 Given a fibrant parameterized spectrum $\ce$ and an object $Y$ of $Spaces_X$,  the   cohomology spectrum $\ce^\bullet (Y)$ is defined to be the  (ordinary) spectrum of sections along $Y$,  $
 \Gamma_Y(r^*\ce)$.    
 This is the spectrum whose $j$-th space
is the mapping space $Map_{Spaces_X}(Y, {\cal E}_j)$.
 The structure maps for ${\cal E}$ yield structure
maps on these mapping spaces.   The cohomology groups are defined to be   the homotopy groups of  $\ce^\bullet (Y)$.
 \end{definition}

 \med
 Here is the general  Poincar\'e duality theorem  that we will need. We are stating it on the spectrum level (rather than on the level of homotopy groups), and in some ways it can be viewed as a twisted version of Atiyah duality \cite{atiyahdual}.     This  result was stated in \cite{umkehr}, but has its origins in \cite[thms.\ A,D]{Klein_dualizing}, \cite[\S5,8]{Klein_dualizing_2}, 
\cite[th.\ 4.9]{pohu}, \cite[prop.\ 2.4]{WW1}).  A complete proof is given in \cite{maysigurd}.

 \med
 
 \begin{theorem}\label{poincare}  Let $M$ be a closed $n$-dimensional manifold, and $\ce \to M$ a fibrant parameterized spectrum over $M$.  There there is an equivalence between the following twisted cohomology and homology spectra,
  $$ \Gamma_M(\ce) \simeq \ce_{-TM}(M)$$. 
\end{theorem}

\noindent \bf Note.  \rm  In this theorem $M$ need not be orientable.  Orientability is required when applying the Thom isomorphism in (co)homology.
\med
\begin{corollary}\label{fiber}  Let $M$ be a closed manifold, and $F \to E \to M$ a fiber bundle.  Then there is an equivalence of spectra,
$$
\Gamma_M(\fsm (E_+)) \simeq E^{-TM}.
$$
\end{corollary}.

\begin{proof}  This follows from the Poincar\'e duality theorem above and Lemma \ref{thom}. \end{proof}  
 
The following  example of this lemma is important in string topology.

\med
\begin{corollary}\label{loop}  Let $M$ be a closed manifold, and let $e : LM \to M$ be the map that evaluates a loop at the basepoint of the circle.
Then there is an equivalence,
$$
\Gamma_M(\fsm (LM_+)) \simeq LM^{-TM}.
$$
\end{corollary}

\med

We now describe how this  result generalizes to the setting of principal bundles over a manifold.    
Let 
 $ 
 P 
 $ 
 be a principal $G$-bundle over $M$.   Indeed, $p: P \to M$ need not be  strictly speaking a principal bundle, but only needs to be a Serre fibration that is fiber homotopy equivalent to a principal bundle.  An example of such a fibration is the following.
 
Let $x_0 \in M$ be a fixed basepoint.  Let $$\cp (M) = \{ \gamma : [0,1] \to M \, \text{such that} \, p(0) = x_0 \}$$  Then evaluating a path at $t=1$ defines the ``path-loop" fibration
$$
\Omega M \to \cp (M) \xr{e} M.
$$
By a construction of Kan,  the based loop space $\Omega M$ is homotopy equivalent to the geometric realization
of a simplicial group, which we call $G_M$.  Notice that the classifying space $BG_M$ is homotopy equivalent to $M$, and indeed the universal principal bundle 
$$
G_M \to EG_M \xr{p} BG_M
$$
is fiber-homotopy equivalent to the path-loop fibration $\Omega M \to \cp (M) \xr{e} M.$

Now given a principal bundle $G \to P \xr{p} B$, let $G \to \pud \to B$ be the corresponding adjoint bundle.  That is,
\begin{equation}\label{pad}
P^{Ad} = P \times_G G^{Ad}
\end{equation}
where $G^{Ad}$ is the topological group $G$  given the action of $G$ by conjugation.   In the case of a universal principal bundle, $G \to EG \to BG$, we have the following well known result.

\begin{proposition} The fibration $G \to EG^{Ad} \to BG$   is fiber-homotopy equivalent to the fibration $G \to L(BG) \xr{e} BG.$
\end{proposition}

Notice that in particular, $EG_M^{Ad} \simeq LM$.  
We therefore have the following generalization of Corollary \ref{loop}.

\begin{theorem}\label{prinstring}  Let $G \to P \to M$ be a principal bundle over a manifold $M$.  Then there is an equivalence,
$$
\Gamma_M(\fsm (\pud_+)) \simeq (\pud)^{-TM}.
$$
\end{theorem}
 This theorem  helps to elucidate the results of \cite{gruhersalvatore} in the following way.
Recall that the adjoint bundle of a principal bundle $P^{Ad} \to M$ is a  fiberwise group.    The multiplication gives the fiberwise suspension spectrum $\fsm (P_+) $  the structure of a fiberwise ring spectrum.  Then the   cohomology spectrum $\Gamma_M(\fsm (\pud))$  is an ordinary ring spectrum.  Just as ordinary Poincar\'e duality sends the cup product in cohomology to the intersection product in homology (up to sign),  the Poincar\'e -Atiyah duality Theorem \ref{poincare}   sends the ring structure in the twisted cohomology spectrum to an intersection product defined via a Pontrjagin-Thom construction, in the twisted homology spectrum.  Applied to this case we get the following.

\begin{proposition}\label{ringequiv} The Poincar\'e- Atiyah equivalence given in Theorem \ref{prinstring}
$$\Gamma_M(\fsm (\pud_+)) \simeq (\pud)^{-TM}$$ preserves the ring multiplications up to homotopy.
\end{proposition}

Motivated by this we  redefine the string topology spectrum $\cs (P)$ to be $\Gamma_M(\fsm (\pud_+))$ with its ring structure. 
 In the particular case of $G_M \to EG_M \to M$,  we write $\cs (EG_M)$ as $\cs (M)$, which we call the string topology spectrum of  $M$.  In this case we have
$$
\cs (M) \simeq LM^{-TM}.
$$

\med
We now point out that the string topology spectrum can be interpreted as an ``endomorphism spectrum" in the following way.   Let $E$ be a ring spectrum and
$E \to \ce \to M$ an $E$-line bundle over $M$.  Consider the parameterized spectrum $End_M (\ce) \to M$ whose fiber at $x \in M$ is the spectrum of $E$-module endomorphisms of the fiber at $x$,  $End^E(\ce_x)$.  Let $End (\ce) = \Gamma_M(End_M (\ce))$ be the spectrum of sections.  The following is now a straightforward exercise. It implies Proposition \ref{stringend} as stated in the introduction.  

\begin{proposition}\label{endo}  Let $G \to P \to M$ be a principal bundle. There is an equivalence of fiberwise ring spectra,  
$$
End_M(\fsm (P_+)) \simeq \fsm (\pud_+).
$$
By applying sections, we have an equivalence of ring spectra,
$$
End (\fsm (P_+)) \simeq \Gamma_M(\fsm (\pud_+)) = \cs (M).
$$
In particular
$$
End(\fsm ((EG_M)_+)) \simeq \cs (M) \simeq LM^{-TM}.
$$
\end{proposition}

\med
Given any ring spectrum $E$ and $E$-line bundle $E \to \ce \to M$,  we observe that the group-like monoid of homotopy automorphisms of $\ce$ are the units of the endomorphism ring, $$ hAut (\ce) = GL_1(End (\ce)).$$   

 Notice that  Corollary \ref{gaugestring} as stated in the introduction now follows immediately from Proposition \ref{endo}.

 
 \section{Gauge groups and the units of string topology}

 In this section we will prove Theorem \ref{main1},  Corollary \ref{cor1}, and Proposition \ref{iota} as stated in the introduction, and then apply them to do a specific calculation.
 
  \begin{theorem}\label{loopaut}  Let $G_M \to EG_M \to BG_M$ be a universal principal bundle  with $EG_M$ contractible.  Let $\Omega M \to LM \xr{e} M$ be the evaluation fibration defined by $e(\gamma) = \gamma (0).$   Then the homotopy type of the gauge group $\cg (EG_M)$ can be described in the following two ways:
\begin{align}
\cg (EG_M) &\simeq \Gamma_M(LM) \notag \\
&\simeq \Omega HE(M) \notag
\end{align} where  $HE(M)$ is the group-like monoid of self-homotopy equivalences of $M$.  
\end{theorem}

\begin{proof}    Given any group $G$, the adjoint bundle of the universal bundle, $G \to EG^{Ad} = EG \times_G G \to BG$ is well known to be fiber homotopy equivalent to the evaluation fibration of the loop space, $\Omega BG \to LBG \xr{e} BG$.  The first statement in the theorem then follows. The fact that $\cg (EG_M) \simeq \Omega HE(M)$ follows from Corollary \ref{loop}, by noticing that $BG_M \simeq M$, 
and then observing that the identity map $id : M \to M \simeq BG_M$ classifies the universal bundle $G_M \to EG_M \to BG_M \simeq M$.  The corollary then implies that $\cg (EG_M) \simeq \Omega Map_{id}(M,M)$.  However the component of the space of self maps of $M$ homotopic to the identity consists of homotopy equivalences.  Furthermore  this is the component of the basepoint ($= id$) in  $HE(M)$.  Thus 
$$
 \cg (EG_M) \simeq \Omega Map_{id}(M,M) = \Omega HE(M).
 $$
 \end{proof}
 
 \med
 We now prove  Theorem \ref{main1}.
 \begin{proof}  
 Let $E$ be a ring spectrum.  Given an $E$-line bundle, $E \to \ce \to M$, recall that  $hAut (\ce) = GL_1(End (\ce))$.    Furthermore, in the discussion prior to the statement of  Proposition \ref{endo},  the endomorphism spectrum $End (\ce)$ is defined to be the section spectrum $\Gamma_M(End_M \ce)$, where $End_M(\ce) \to M$ is the parameterized spectrum whose fiber at $x \in M$ is the spectrum of $E$-module endomorphisms of the fiber $\ce_x$.   Now consider the corresponding fiber bundle of infinite loop spaces 
given by taking the zero spaces of the fibrant model of $End_M(\ce)$:
$$
\Omega^\infty E \to    \Omega^\infty_MEnd_M(\ce) \to M.
$$
  By restricting to path components of these $E$-module endomorphisms that are equivalences, we get a subbundle, which we will call $ \cg L_1(\ce)$:
$$
GL_1(E) \to  \cg L_1(\ce)  \to M.
$$
 Furthermore, $hAut (\ce) = \Gamma_M(\cg L_1(\ce))$.  Notice, however, that this bundle has another description.  The homotopy class of map $f_\ce : M \to BGL_1(E)$ that classifies the $E$-line bundle $\ce$,  also classifies a principal $GL_1(E)$-bundle over $M$:  
 $$
 GL_1(E) \to P_\ce \to M.
 $$
 (Here we are replacing the group-like topological monoid $GL_1(E)$ by a topological group.)  The relationship between $P_\ce$ and the parameterized spectrum $\ce$ is that $$\Omega^\infty_M \ce = (P_\ce) \times_{GL_1(E)} \Omega^\infty E.$$   (See \cite{units} and \cite{maysigurd} for details.)    
 
 Now consider the corresponding adjoint bundle, $GL_1(E) \to P_\ce^{Ad} \to M$.  As pointed out before, this bundle can be viewed as  being the bundle of $GL_1(E)$-equivariant automorphisms of $P_\ce$ covering the identity map on $M$.  From this viewpoint it becomes clear that there is a map of fibrations   $ P_\ce^{Ad} \to \cg L_1(\ce)$ over $M$ which is an equivalence on the fibers.     Therefore their spaces of sections are equivalent:
 $$
 hAut (\ce) = \Gamma_M(\cg L_1(\ce)) \simeq \Gamma_M( P_\ce^{Ad}) = \cg (P_\ce).
 $$
 Theorem  \ref{main1} now follows from the Atiyah Bott equivalence (\ref{abgauge}) applied to the principal bundle $GL_1(E) \to P_\ce \to M.$
 \end{proof}
 
 \med
 We note that Corollary \ref{cor1} as stated in the introduction now follows immediately.   We now prove Proposition \ref{iota}.

\begin{proof}
Suppose $G \to P \to M$ is a principal bundle.  Consider the homomorphism $\iota : G \to GL_1(\Sigma^\infty (G_+))$ described earlier.  This extends to a map
 of principal bundles,
\begin{equation}\label{iotaprin}
 \begin{CD}
 G    @>>>  P  @>>>   M \\
 @V\iota VV   @V\iota VV  @VV=V \\
 GL_1(\Sigma^\infty (G_+))   @>>> P_{\fsm (P_+)}  @>>> M
 \end{CD}
 \end{equation} where $ P_{\fsm (P_+)}$ is the principal $GL_1(\Sigma^\infty (G_+)) $-bundle associated to the $\Sigma^\infty (G_+)$-line bundle
 $\fsm (P_+) \to M$ as described above.    We therefore get an induced map of adjoint bundles

 \begin{equation}\label{adjoint}
 \begin{CD}
 G    @>>>  P^{Ad}  @>>>   M \\
 @V\iota VV   @V\iota VV  @VV=V \\
 GL_1(\Sigma^\infty (G_+))   @>>> P_{\fsm (P_+)}^{Ad}  @>>> M
 \end{CD}
 \end{equation}
 By taking sections we get the  map of gauge groups,  $\sigma : \cg (P) \xr{\iota} \cg (P_{\fsm (P_+)}) \simeq hAut (\fsm (P_+)).$
 Since the Atiyah-Bott proof of (\ref{abgauge})   is natural, this gives the following commutative diagram
 $$
 \begin{CD}
 \cg (P) @>\iota >>   \cg (P_{\fsm (P_+)})  \\
 @V\simeq VV    @VV\simeq V \\
 \Omega Map_P(M,  BG)    @>>\iota >   \Omega Map_{\fsm (P_+)}(M, BGL_1(\Sigma^\infty (G_+))
 \end{CD}
 $$
 This implies the statement of Proposition \ref{iota}.  \end{proof}

 \bg
We now compute an explicit example.
 Consider the Hopf fibration $\ch$ over $S^4$: 
 $$
 S^3 \to S^7 \to S^4.
 $$
 This is a principal $SU(2)$-bundle.         Our goal is to compute the homotopy type of $hAut^b(\Sigma^\infty_{S^4} (\ch_+))$, the space of \sl based \rm homotopy automorphisms of the fiberwise suspension spectrum.      This is  defined to be the subgroup of $hAut (\Sigma^\infty_{S^4} (\ch_+))$ consisting of those homotopy automorphisms that are homotopic to the identity  on the fiber spectrum  at the basepoint $\ce_{x_0}$.
 Notice that there is a fibration sequence
\begin{equation}\label{based}
 hAut^b(\ce) \to hAut (\ce) \xr{ev} GL_1(E)
\end{equation}
 where $ev : hAut (\ce) \to GL_1(E)$ evaluates a homotopy automorphisms on the fiber over the basepoint.  From this fibration and Theorem \ref{main1}, it is easy to conclude that there is an equivalence,
 \begin{equation}\label{basedmain1}
  hAut^b(\ce) \simeq  \Omega Map_\ce^b(M, BGL_1(E))
   \end{equation}  where $Map_\ce^b$ denotes the space of based maps that classify the $E$-line bundle $\ce$.
   
   To state the result explicitly, we adopt the following  customary notation. For a based space $X$ let $QX = \lim_{n\to \infty} \Omega^n\Sigma^n X$.
   
   \begin{theorem}\label{hopf}  The group of units of the string topology spectrum of the  Hopf fibration $\ch$  has the following homotopy type (as group-like $A_\infty$-spaces):
   $$
   GL_1(\cs (\ch)) \simeq \Omega QS^0 \times \Omega^4QS^0.
   $$  In particular $\pi_0(   GL_1(\cs (\ch))) \cong \bz/2$.  
       \end{theorem}
   \begin{proof}  By observation (\ref{basedmain1})  $$GL_1(\cs (\ch)) = hAut^b(\fsm \ch_+) \simeq \Omega Map^b_\iota (S^4,    BGL_1(\Sigma^\infty (S^3_+)).$$  This based mapping space is a path component of $\Omega^4 BGL_1(\Sigma^\infty (S^3_+)) \simeq \Omega^3GL_1(\Sigma^\infty (S^3_+)).$  Since all path components of a based loop space are homotopy equivalent, this says that
   $$GL_1(\cs (\ch)) \simeq \Omega^4(GL_1(\Sigma^\infty (S^3_+)))$$ as group-like $A_\infty$-spaces.  Now $GL_1(\Sigma^\infty (S^3_+))$ consists of those path components of $Q(S^3_+)$ corresponding to the units in the ring $\pi_0(Q(S^3_+)) \cong \bz$.   Again, since $Q(S^3_+))$ is a loop space all of its path components are homotopy equivalent.  Thus $\Omega^4(GL_1(\Sigma^\infty (S^3_+))) = \Omega^4(Q(S^3_+))).$   We therefore have
\begin{align}
   GL_1(\cs (\ch))  \simeq \Omega^4(Q(S^3_+)) &\simeq \Omega^4(QS^3) \times \Omega^4QS^0  \notag \\
   &\simeq  \Omega QS^0 \times \Omega^4QS^0 \notag
   \end{align}
    \end{proof}
   
   We remark that this theorem says that   $hAut(\fsm (\ch))$  is homotopy equivalent to $\Omega QS^0 \times \Omega^4QS^0$.   Now the gauge group of the principal bundle $\cg(\ch)$ has the homotopy type of $\Omega Map_u^b(S^4, BS^3)$ where $u : S^4 \to BS^3 \cong \bh \bp^\infty$ generates $\pi_4(\bh \bp^\infty) \cong \bz$.   This space is equivalent to $\Omega^4S^3$.  Proposition \ref{iota} then implies that the map $\sigma : \cg(\ch)  \to hAut (\fsm (\ch))$ is given by the stabilization map
   $$
   \sigma : \Omega^4S^3 \to \lim_{n\to \infty} \Omega \Omega^n S^n  = \Omega QS^0 \hk \Omega Q S^0 \times \Omega^4 QS^0.
   $$

 \section{Connections on an $E$-line bundle}
 
 The goal of this section is to describe what is meant by a ``connection" on an $E$-line bundle, study the category of such connections, and to prove Theorem \ref{main2}.

 \med
  Recall that for a principal bundle $G \to P \to M$, if $\ca (P)$ denotes the space of connections on $P$, then $\ca (P)$ has a natural action of the gauge group $\cg (P)$.   The following was observed by Atiyah and Bott \cite{atiyahbott}.
 
 \begin{proposition} \label{principalconnect}The homotopy orbit space, 
 $ 
E\cg (P) \times_{\cg (P)}  \ca (P)$ which we denote by $\ca (P) // \cg (P)$, is a model for the classifying space,
 
$$\ca (P) //\cg (P) \simeq B\cg (P) \simeq Map_P (M, BG).$$
\end{proposition}
 
 Theorem \ref{main2} is an analogous result, in the context of  connections on an $E$-line bundle.
 The appropriate way to think of a connection $\alpha$  on an $E$ -line bundle   $E \to \ce \to M$  is as an $E$-linear transport operator.  
 Namely, for every path $\gamma$ between points $x$ and $y$ in $M$,   the connection $\alpha$ will define an $E$-module equivalence of the fiber spectra 
 $$
 \alpha (\gamma ) : \ce_x \xr{\cong} \ce_y.
 $$
 Furthermore these equivalences   respect ``gluing of paths".  
 
 To say this more precisely, we use functorial language.
 Given a connected space $X$, let $Path (X)$ denote the \sl path category \rm of $X$.  This is a topological category, whose objects are the points in $X$,  with the topology inherited from $X$, and the morphisms from $x \to y$ are the ``Moore paths",  namely pairs, $(r, \gamma)$,  where $r > 0$ and $\gamma : [0,r] \to M$ is a path with $\gamma (0) = x$ and $\gamma (r) = y$.  These path spaces inherit the natural topologies and they have the properties that if $Mor_{Path(X)}$ is the full space of morphisms, then the source and target maps, $s, \, t : Mor_{Path(X)} \to X$ are Serre fibrations with contractible fibers.  The following is a standard result.
 
 \begin{proposition}\label{path}  The geometric realization of the path category has the weak homotopy type of $X$,
 $$
 |Path(X)| \simeq X.
 $$
 \end{proposition}
 
 \med
 Now we consider another category.   Let $E$ be a ring spectrum, and let $E-line$ be the category of ``$E$-lines", as in \cite{5author}.  The objects  of $E-line$ are free, rank-one cofibrant and fibrant $E$-modules, and the morphisms are equivalences of $E$-modules. Actually in \cite{5author} $E-line$ was defined to be an ``$\infty$-category".  In any case as observed in \cite{5author}, one has the following:
 
 \med
 \begin{proposition} \label{line}   The geometric realization of the category $E-line$ is a model for the classifying space of the group of units,
 $$
 |E-line| \simeq BGL_1(E).
 $$
 \end{proposition}

 \med
 Intuitively,  a continuous functor $\phi : Path(M) \to E-line$ determines an    $E$-line bundle over $M$ which we call $E^\phi$,  together with an $E$-linear transport operator (``connection").  Namely, the fiber spectrum   $E^\phi_x$  at a point $x \in M$ is $\phi (x)$,  and if $\gamma : x \to y$ is a path in $M$ between points $x$ and $y$, then on the level of morphisms, $\phi (\gamma ) : E^\phi (x) \to E^\phi (y)$ is an equivalence.    
 
 \med
 For the  purposes of making this precise we actually will use  a somewhat more general notion of a connection on an $E$-line bundle.   A functor $\phi$ as above has the property that if $\gamma : x \to y$ is a path from $x$ to $y$ in $M$, and $\alpha : y \to z$ is a path from $y$ to $z$, then $\phi (\alpha) \circ \phi (\gamma) :E^\phi (x) \to E^\phi (y) \to E^\phi (z)$ equals $\phi (\gamma * \alpha)$, where $\gamma * \alpha : x \to z$ is the glued path.  What we will require instead is that this strict equality is replaced by  a coherent homotopy between these $E$-linear transport operators.    We accomplish this by the use of $\infty$-categories and functors between them.  For this we follow and use the constructions and results of \cite{5author}.  There, and here, the notion of an $\infty$-category is that of a   ``quasi-category" in the sense of Joyal \cite{joyal}.
 
 Recall that an $\infty$-category in this sense is a simplicial set with additional structure that satisfy a variety of properties.   The  objects are the zero-simplices and the morphisms are the one-simplices. Given a space $X$, the singular set $Sing (X)$ is an $\infty$-category (indeed an $\infty$-groupoid) which is equivalent to the topological category  $Path (X)$.  Notice that the objects of  $Sing (X)$ are points in $X$, and the morphisms are paths.  As mentioned above,  $E-line$ was defined in \cite{5author} to be an $\infty$-category.    
 
 As described in \cite{5author}, an $\infty$-functor $\phi : Sing (X) \to   E-line$ defines an $E$-line bundle $E^\phi$ on $X$.  Notice that on the level of morphisms (one-simplices), $\phi$ defines $E$-linear transport operators along the paths.   On the level of $2$-simplices $\phi$ defines homotopies between the operator applied to a glued path and the composition of the operators.  On the level of higher simplices one sees higher homotopies.  
 We therefore make the following definition.

 \begin{definition}\label{connection}  Given an  $E$-line bundle $E  \to\ce \to M$, define the $\infty$-category of connections on $\ce$, $Conn (\ce),$ to be the $\infty$-functor category, consisting of  functors $\phi : Sing (M) \to E-line$ so that $E^\phi \cong \ce$. 
 \end{definition}
 
 \med
 Notice that in the setting of a principal bundle $P$  one can define the category of connections in an analogous way.    Namely, let $G-line$ be the category of free, transitive $G$-spaces, and equivariant homeomorphisms between them.  We may then think of a connection as $G$-equivariant  transport operator which is given by a continuous functor, $\phi : \cp(M) \to  G-line$ with the property that $\phi (x) = P_x$.    Notice that the gauge group
 $\cg (P)$ is exactly the space of natural transformations between such functors.  If $Conn (P)$ is the category of connections viewed as such functors, and natural transformations (gauge transformations) between them,  Proposition \ref{principalconnect} of Atiyah and Bott can then be interpreted as the following.
 
 \begin{proposition}\label{connp} The geometric realization of the category of connections is a model for the classifying space of the gauge group
 $$
 |Conn (P)| \simeq B\cg (P) \simeq Map_P(M, BG).
 $$
 \end{proposition}
 
  In the category  $Conn (\ce)$ of  connections on an  $E$-line bundle $E\to \ce \to M$,   $hAut (\ce)$ acts as natural transformations between connections, and therefore   can be viewed as    morphisms in an analogous way.     From this perspective Theorem \ref{main2} is  the analogue of Propositions \ref{principalconnect} and \ref{connp}.  We now prove Theorem \ref{main2}.

 \begin{proof}     The $\infty$-category $Conn (\ce)$ is  the category of $\infty$-functors,   $Fun_\ce (Sing(M),  E-line)$.    Both $Sing(M)$ and $E-line$ are $\infty$-groupoids.  This implies that these $\infty$-categories can be regarded as spaces.  More precisely, the geometric realization map
 $$
|Conn (\ce)| = |Fun_\ce (Sing(M),  E-line)|   \to  Map_\ce(|Sing(M)|, |E-line|)
$$
is a weak homotopy equivalence.  By Proposition \ref{line} this space is equivalent to $Map_\ce(M, BGL_1(E))$.    But by Theorem \ref{main1} this is homotopy equivalent to $BhAut (\ce)$.   This proves the theorem.
 \end{proof}

 \section{Loop groups and string topology}
 
 The goal of this section is to study the stabilization map $\sigma :  \cg (P) \to  hAut (\fsm (P))$ in the case when $P$ is universal.    
  Recall from Theorem \ref{loopaut} that   $\cg (P) \simeq \Omega HE(M)$.  
 Now  if $K$ is a Lie group acting on a manifold $M$, this defines a homomorphism $K \to HE(M)$.  Then the composition
\begin{equation}\label{rho}
\rho:  \Omega K \to  \Omega HE (M) \simeq \cg (\cp (M)) \xr{\sigma} \cg (\fsm (\cp(M))   \simeq GL_1(\cs (M)) \simeq GL_1(\ltm)
\end{equation}
is a natural representation of the loop group on the string topology spectrum, as stated in Proposition \ref{equivariant}  in the introduction.  This is a  highly nontrivial representation.  Below we will compute a particular example of this representation on the level of homology.   But first, we consider this representation in the following way.  Notice the map $\rho : \Omega K \to GL_1(\ltm)$ defines a map of ring spectra, which by abuse of notation we also call $\rho$,

\begin{equation}\label{rho2}
\rho : \Sigma^\infty (\Omega K_+) \to \ltm.
\end{equation}
This map of spectra we can define explicitly in the following way.  Consider the map  
\begin{align}\label{evaluate}
\zeta : \Omega K \times M &\to LM  \\
\gamma \times x &\to \gamma_x
\end{align}
where $\gamma_x(t) = \gamma (t)\cdot x$.  Pulling back the virtual bundle $-TM$ and taking Thom spectra gives an induced map
$$
\Omega K _+ \wedge M^{-TM} \xr{\gamma} \ltm.
$$
Now $M^{-TM}$ is a ring spectrum, equivalent to the Spanier-Whitehead dual, $Map (M, S^0)$, and is equipped with a unit $S^0 \xr{u} M^{-TM}$.  This unit map may be viewed as the Pontrjagin-Thom construction for the normal bundle of an embedding $e : M \hk R^N$.  That is, given such an embedding with normal bundle $\nu_e$, one has Pontrjagin-Thom construction $S^n = \br^n \cup \infty \to M^\nu = \Sigma^n M^{-TM}$.   The unit map $u : S^0 \to M^{-TM}$ is homotopic to the desuspension of this map. We then have the composition
\begin{equation} 
\bar \rho : \Sigma^\infty (\Omega K_+) \wedge S^0 \xr{1 \wedge u}   \Sigma^\infty (\Omega K_+) \wedge  M^{-TM} \xr{\zeta}  \ltm.
\end{equation}

If $M$ has an orientation, then applying the Thom isomorphism defines a map
$$
\bar\rho_* : H_q( \Omega K)\to H_{q+n}(LM).
$$
This  map taken with rational coefficients was studied by Felix and Thomas  in \cite{felixthomas}.

\med
 We now discuss an explicit example.
Consider the standard action of $SU(2)$ on $S^4$ defined by rotation around the equator.  By  Proposition \ref{equivariant} this defines a representation
of the loop group, 
$$
\rho : \Omega SU(2) \to GL_1({LS^4}^{-TS^4})
$$
and by the above discussion, an induced map of ring spectra,
$$
\bar \rho : \Sigma^\infty (\Omega SU(2)_+) \to {LS^4}^{-TS^4}.
$$
We recall the stable homotopy type of each of these ring spectra.  The ring spectrum $$\Sigma^\infty (\Omega SU(2)_+) = \Sigma^\infty (\Omega S^3_+) \simeq \Sigma^\infty (\bigvee_{k=0}^\infty S^{2k} )$$
is the free associative ring spectrum generated by $S^2$. This is the statement of the classical James-Milnor splitting. In particular, $$H_*(\Omega SU(2)) \cong \bz [\sigma]$$ where $\sigma \in H_2(\Omega S^3) \cong \bz$ is the generator.     On the other hand, since the tangent bundle of  $S^4$ is stably trivial, the splitting result of Carlsson and Cohen \cite{CC} \cite{C2} implies the following.
$$
{LS^4}^{-TS^4} \simeq \Sigma^\infty (\Sigma^{-4}(LS^4_+)   \simeq  \Sigma^{-4} \left(\bigvee_{q=0}^\infty S^1_+\wedge_{\bz/q} S^{3q}\right).
$$
In homology, the ring structure of $H_*({LS^4}^{-TS^4})$ was calculated in \cite{CJY} as 
 
 $$H_*({LS^4}^{-TS^4}) \cong  \Lambda [b] \otimes \bz[a, v]/(a^2, ab, 2av)$$
 where
 $a \in H_{-4}({LS^4}^{-TS^4}),  b \in H_{-1}({LS^4}^{-TS^4}),$ and $v \in H_6({LS^4}^{-TS^4})$.
 
 \med
 The following is proved by tracing through the definitions, and shows the nontriviality of representation $\rho$.  
 
 \begin{proposition}
 The ring homomorphism
\begin{align}
\bar \rho : H_*(\Omega SU(2))    &\to H_*({LS^4}^{-TS^4}) \notag \\
\bz [\sigma] &\to \Lambda [b] \otimes \bz[a, v]/(a^2, ab, 2av) \notag
\end{align}
is determined by the property that 
$$
\bar \rho (\sigma) = av.
$$
\end{proposition}

 \section{String topology as the linearization of the gauge group}
 
 The goal of this section is to prove Theorem \ref{main3} as stated in the introduction.
 \begin{proof}  Let $G \to P \to M$ be a principal bundle.  Consider the gauge group functor $\cg : Spaces_M \to Spectra$ and the string topology functor $\cs : Spaces_M \to Spectra$ as defined in the introduction prior to the statement of the theorem.  We first prove that $\cs$ is a linear functor. 
 
  In \cite{umkehr} Cohen and Klein proved a fibered version of Brown representability (see Theorem A.4 of that paper).  Although the  terminology used there is a bit different, that theorem says that  a contravariant homotopy functor $\cf : Spaces_M \to Spectra$  is  linear   if and only if there is a parameterized spectrum $\ce \to M$ such that $\cf (g : Y \to M) \simeq \Gamma_Y(g^*(\ce))$.  If we let $\ce = \fsm (\pud_+)$, then this implies that the string topology functor is linear.
 
 We now show that the string topology functor $\cs$ is the \sl universal \rm linear approximation to the gauge theory functor,  $\cg (g : Y \to M) = \Sigma^\infty ((\Gamma_Y(g^*(P^{Ad}))_+)$.  Let $\cl : Spaces_M \to Spectra$ be a linear homotopy functor admitting a natural transformation
 $$
 \tau : \cg \to \cl.
 $$
By the fiberwise Brown representability theorem there is a parameterized spectrum $\ce \to M$ such that $\cl (g:Y \to M) = \Gamma_Y(g^*(\ce))$.   Here we are assuming that $\ce$ is  a fibrant parameterized spectrum.

So for every object $g:Y \to M$ in $Spaces_M$,  $\tau$ induces a map of spectra,
$$
\tau_Y:  \Sigma^\infty (\Gamma_Y(g^*(P^{Ad}))_+)   \to \Gamma_Y(g^*(\ce)).
$$

 Let $\Omega^\infty_M\ce \to M$ be the fibration where the total space is the zero-space of the fibrant parameterized spectrum. 
  Then on the level of spaces, $\tau$ induces a map of infinite loop spaces,
  $$
  \tau_Y :  \Omega^\infty \Sigma^\infty (\Gamma_Y(g^*(P^{Ad}))_+)   \to \Gamma_Y(g^*(\Omega^\infty_M\ce)).
 $$
 Restricting to $\Gamma_Y(g^*(P^{Ad})) \subset \Omega^\infty \Sigma^\infty (\Gamma_Y(g^*(P^{Ad}))_+) $,  we have an induced map of spaces,
 $$
 \tilde \tau_Y  :   \Gamma_Y(g^*(P^{Ad}))  \to   \Gamma_Y(g^*(\Omega^\infty_M\ce)).
 $$
 
 Now consider the object $p: P^{Ad} \to M$ in $Spaces_M$.  We then get a map
 $$
 \tilde \tau_{P^{Ad}} :   \Gamma_{P^{Ad}}(p^*(P^{Ad})) \to \Gamma_{P^{Ad}}(p^*(\Omega^\infty_M\ce)).
 $$
 Written in different notation, this is a map
 $$
  \tilde \tau_{P^{Ad}} : Map_M(P^{Ad}, P^{Ad}) \to Map_M(P^{Ad}, \Omega^\infty_M \ce)),
  $$
 where $Map_M$ means morphisms in the category $Spaces_M$.  
 
 Now consider the image under this map of the identity map, $  \tilde \tau_{P^{Ad}} (id) : P^{Ad} \to \Omega^\infty_M \ce).$  Since this is a morphism in $Spaces_M$ this defines a map (well defined up to homotopy) of fibrations

 $$
 \begin{CD}
 P^{Ad}   @>\tilde \tau_{P^{Ad}} (id) >>  \Omega^\infty_M\ce  \\
 @VVV    @VVV \\
 M  @>>= > M.
 \end{CD}
 $$  This map extends to the universal fiberwise infinite loop space  over $M$,
 
 $$
 \begin{CD}
  \Omega^\infty_M\Sigma^\infty_M (P^{Ad}_+)   @>\tau_{P^{Ad}} (id) >>  \Omega^\infty_M\ce  \\
 @VVV    @VVV \\
 M  @>>= > M.
 \end{CD}
 $$

 But this is equivalent to a map of fiberwise spectra over $M$,
 $$
 \begin{CD}
 \Sigma^\infty_M (P_+)   @>\tau_{P^{Ad}} (id) >>  \ce  \\
 @VVV    @VVV \\
 M  @>>= > M.
 \end{CD}
 $$
 Now $\fsm (\pud)$ is the parameterized spectrum representing the string topology functor, and $\ce$ is the parameterized spectrum representing the linear functor $\cl$.  This therefore defines a natural transformation $\bar \tau : \cs \to \ce$.  By construction it is clear that this map satisfies the   universality property in the statement that the string topology functor $\cs$ is the \sl linearization \rm of the gauge theory functor $\cg$.   \end{proof}


\begin{thebibliography}{99}{ 
\bibitem{units}M. Ando, A. J. Blumberg, D. J. Gepner, M. J. Hopkins, and C. Rezk.	\emph{Units of ring spectra and Thom spectra}   preprint, arXiv:0810.4535 

\bibitem{5author} M. Ando, A. J. Blumberg, D. Gepner, M. J. Hopkins, and C. Rezk, \emph{An $\infty$-categorical approach to $R$-line bundles, $R$-module Thom spectra, and twisted $R$-homology}, preprint (2012). 

\bibitem{atiyahdual}M.F. Atiyah, \emph{Thom complexes },
 Proc. London Math. Soc. \textbf{ (3) }, no. 11 (1961),  291--310.

   \bibitem{atiyahbott}M. Atiyah, and R. Bott \emph{The Yang-Mills equations over Riemann surfaces}, Phil. Trans. R. Soc. Lond. A  \bf 308, \rm 523-615 (1982)  
 
 \bibitem{CC} G. Carlsson and R.L. Cohen \emph{The free loop space and the cyclic groups,}  Commentari Math. Helvetici. \bf 62 \rm (1987), 423 - 449.
 
 \bibitem{C2} R.L. Cohen, \emph{ The free loop space of a suspension,  }Springer Lecture Notes \bf 1286
\rm (1987), 193 -
207.
 
 \bibitem{chassullivan}  M. Chas   and D. Sullivan, \emph{String Topology}.
  preprint: math.GT/9911159.
  

  \bibitem{cohenjones} R.L. Cohen and J.D.S. Jones, \emph{ A homotopy theoretic realization 
of string topology},  
  Math. Annalen, \bf  vol. 324,  \rm 773-798 (2002).  preprint: math.GT/0107187 

\bibitem{CJY}R.L. Cohen, J.D.S. Jones, and J. Yan, \emph{The loop homology algebra of spheres and projective spaces}, Progress in Mathematics \bf  Vol. 215 \rm (2003) 77-92.  preprint: http://arxiv.org/pdf/math/0210353.pdf


\bibitem{umkehr}  R.L. Cohen and J.R. Klein, \emph{Umkehr maps},    Homology, Homotopy, and Applications,  \bf vol. 11 \rm (1),  (2009), 17-33.  preprint:
 arXiv:0711.0540 

 

\bibitem{felixthomas}Y. F\'elix  and J -C  Thomas,  \emph{Monoid of self-equivalences and free loop spaces} . Proc. Amer. Math. Soc. 132 (2004), no. 1, 305Ð312
 


\bibitem{goodwillie} T. Goodwillie, \emph{ Calculus. II. Analytic functors.},   K-Theory 5 (1991/92), no. 4, 295Ð332.

\bibitem{gruhersalvatore} K. Gruher and P. Salvatore,  \emph{Generalized string topology operations} Proc. Lond. Math. Soc. (3) \bf 96 \rm    (2008),   78Ð106. 


 
 \bibitem{pohu}%
Hu, P., Duality for smooth families in equivariant stable homotopy theory.
\newblock {\it Asterisqu\'e\rm } {\bf 285}, 2003

\bibitem{joyal} A. Joyal,  \emph{Quasi-categories and Kan complexes}, J. Pure Appl. Algebra, {\bf 175} (2002), 207-222.  Special volume celebrating the 70th birthday of Max Kelley.
 


\bibitem{Klein_dualizing}%
Klein, J.R., The dualizing spectrum of a topological group.
\newblock {\it Math. Annalen\rm } {\bf 319}, 421--456 (2001)

\bibitem{Klein_dualizing_2}%
Klein, J.R.,
The dualizing spectrum. II.  
\newblock {\it Algebr. Geom. Topol.}  {\bf 7}  109--133 (2007)


 \bibitem{cary} C. Malkiewich,  \emph{Stanford University PhD thesis}, in preparation.  See \emph{A tower connecting gauge groups to string topology}, (2012), preprint http://arxiv.org/pdf/1209.1778.pdf


\bibitem{maysigurd} J. P. May and 
 J.  Sigurdsson,  {Parametrized homotopy theory}.
\newblock Mathematical Surveys and Monographs, vol. 132,
\newblock Amer. Math. Soc., 2006
\newblock Mathematical Surveys and Monographs, {\bf vol. 132},
\newblock Amer. Math. Soc., 2006


\bibitem{WW1} Weiss, M., Williams, B.: {Automorphisms of manifolds and algebraic $K$-theory.  I}. \newblock {\it $K$-theory \rm } {\bf 1}, 575--626 (1988)






     
         }\end{thebibliography}
 \end{document}